\documentclass[10pt,reqno]{amsart}
  \usepackage{geometry}
\usepackage[utf8]{inputenc}
\usepackage{amsmath,amsthm,amssymb}
\usepackage{mathrsfs}
\usepackage{enumerate}
\usepackage{graphicx}
\usepackage{xfrac}
\usepackage{subfiles}
\usepackage{tikz-cd}
\usepackage{comment}
\usepackage{mathtools}
\usepackage{enumitem}
\usepackage{dynkin-diagrams}
\usepackage{cancel}
\usepackage{soul}
\usepackage{accents} % use if Aball and Bball are defined below
\usepackage[colorlinks=true,hidelinks]{hyperref}
    \AtBeginDocument{}
\usepackage[capitalise]{cleveref}
\usepackage{tikz} % for tikz pictures
\usetikzlibrary{arrows}
\usetikzlibrary{patterns}
\usepackage{pgfplots} % for graphing functions in tikz
\usetikzlibrary{decorations.markings} % use for arrows in tikz
\usetikzlibrary{3d,calc} %3d plotting in tikz
\usepackage{pifont}
\DeclareFontFamily{OT1}{pzc}{}
  \DeclareFontShape{OT1}{pzc}{m}{it}{<-> s * [1.200] pzcmi7t}{}
  \DeclareMathAlphabet{\mathpzc}{OT1}{pzc}{m}{it}
\usepackage{fancyhdr}
\usepackage{diagrams}

\let\oldtocsection=\tocsection
\let\oldtocsubsection=\tocsubsection
\renewcommand{\tocsection}[2]{\hspace{0em}\oldtocsection{#1}{#2}}
\renewcommand{\tocsubsection}[2]{\hspace{3em}\oldtocsubsection{#1}{#2}}

\title{Higher Fano Manifolds}
\author[Carolina Araujo et al.]{Carolina Araujo}
\address{Carolina Araujo, IMPA, 
Estrada Dona Castorina 110,
22460-320 Rio de Janeiro, Brazil}
\email{caraujo@impa.br}

\author[]{Roya Beheshti}
\address{Roya Beheshti, Department of Mathematics \& Statistics, 
Washington University in St. Louis, St. Louis, MO, 63130}
\email{beheshti@wustl.edu}

\author[]{Ana-Maria Castravet}
\address{Ana-Maria Castravet, Universit\'e Paris-Saclay, UVSQ, Laboratoire de Math\'ematiques de Versailles, 78000, Versailles, France}
\email{ana-maria.castravet@uvsq.fr}

\author[]{Kelly Jabbusch}
\address{Kelly Jabbusch, Department of Mathematics \& Statistics, University of Michigan--Dearborn, 4901 Evergreen Rd, Dearborn, Michigan, 48128, USA}
\email{jabbusch@umich.edu}

\author[]{Svetlana Makarova}
\address{Svetlana Makarova, Department of Mathematics, University of Pennsylvania, 209 S 33rd St, Philadelphia, PA 19104, USA}
\email{murmuno@sas.upenn.edu}

\author[]{Enrica Mazzon}
\address{Enrica Mazzon, Max-Planck-Institut f{\"u}r Mathematik, Vivatsgasse 7, 53111, Bonn, Germany}
\email{e.mazzon15@alumni.imperial.ac.uk}

\author[]{Libby Taylor}
\address{Libby Taylor, Stanford University, 380 Serra Mall, Stanford, CA 94305, USA}
\email{lt691@stanford.edu}

\author[]{Nivedita Viswanathan}
\address{Nivedita Viswanathan, School of Mathematics, The University of Edinburgh, Edinburgh, EH9 3FD, UK}
\email{Nivedita.Viswanathan@ed.ac.uk}

\date{}

\newcommand{\Z}{\mathbb{Z}}

\newcommand{\PS}{\mathbb{P}}
\newcommand{\R}{\mathbb{R}}
\newcommand{\C}{\mathbb{C}}
\newcommand{\Gm}{\mathbb{G}_{\!m}}
\newcommand{\PP}{\mathbb{P}}

\renewcommand{\P}{\mathbb{P}}

\newcommand{\X}{\mathcal{X}}

\renewcommand{\O}{\mathcal{O}}

\renewcommand{\P}{\mathbb{P}}

\DeclareFontFamily{U}{wncy}{}
    \DeclareFontShape{U}{wncy}{m}{n}{<->wncyr10}{}
    \DeclareSymbolFont{mcy}{U}{wncy}{m}{n}
    \DeclareMathSymbol{\Sha}{\mathord}{mcy}{"58} 

\newcommand{\rat}[0]{\operatorname{RatCurves}^n}
\newcommand{\chow}[0]{\operatorname{Chow}}

\newcommand{\Ob}{\mathcal{O}}

\newcommand{\Hom}{\ensuremath{\operatorname{Hom}}}

\newcommand{\ch}{\ensuremath{\operatorname{ch}}}
\newcommand{\OGr}{\ensuremath{\operatorname{OG}}}
\newcommand{\SGr}{\ensuremath{\operatorname{SG}}}

\newcommand{\Pic}{\ensuremath{\operatorname{Pic}}}

\newcommand{\Gr}{\ensuremath{\operatorname{Gr}}}

\renewcommand{\epsilon}{\varepsilon}

\newtheorem{thm}{Theorem}[section]
\newtheorem{prop}[thm]{Proposition}
\newtheorem*{thm1}{Theorem}

\newtheorem{lem}[thm]{Lemma}
\newtheorem{defn}[thm]{Definition}
\newtheorem{cor}[thm]{Corollary}
\newtheorem{conj}[thm]{Conjecture}
\newtheorem{ex}[thm]{Example}
\newtheorem{rem}[thm]{Remark}
\newtheorem{prob}[thm]{Problem}
  \let\oldrem\rem
  \renewcommand{\rem}{\oldrem\normalfont}
\newtheorem*{question}{Question}
  \let\oldex\ex
  \renewcommand{\ex}{\oldex\normalfont}
\theoremstyle{proof}
\numberwithin{equation}{section}

\newcommand{\Fthree}{$\mathfrak{F}_3$}

\newcommand{\gfr}{\mathfrak{g}}
\newcommand{\pfr}{\mathfrak{p}}
\newcommand{\tfr}{\mathfrak{t}}

\begin{document}

\maketitle

\begin{abstract}
In this paper we address Fano manifolds with positive higher Chern characters. They
are expected to enjoy stronger versions of several of the nice properties of Fano manifolds. For instance, they should be covered by higher dimensional rational varieties, and families of higher Fano manifolds over higher dimensional bases should admit meromorphic sections (modulo the Brauer obstruction).
Aiming at finding new examples of higher Fano manifolds, we investigate 
positivity of higher Chern characters of rational homogeneous spaces.
%We find new examples of Fano manifolds having positive second Chern character, and 
We determine which rational homogeneous spaces of Picard rank $1$ have positive second Chern character, and show that the only rational homogeneous spaces of Picard rank $1$ having positive second and third Chern characters are projective spaces and quadric hypersurfaces. 
We also classify Fano manifolds of large index having positive second and third Chern characters.
We conclude by discussing conjectural characterizations of projective spaces and complete intersections in terms of these higher Fano conditions. 
\end{abstract}

\tableofcontents

\section{Introduction}

Fano manifolds are complex projective manifolds having positive first Chern class $c_1(T_X)$.
Examples of Fano manifolds include projective spaces, smooth complete intersections of low degree in projective spaces, and rational homogeneous spaces. 
The positivity condition on the first Chern class has far reaching geometric and arithmetic implications, making Fano manifolds a central subject in algebraic geometry. 
To illustrate the special role played by Fano manifolds, we briefly discuss some of their nice properties. 

\medskip

Hypersurfaces of low degree in projective spaces provide basic examples of Fano manifolds. 
The first Chern class of a smooth hypersurface $X_d \subset \PP^{n+1}$ of degree $d\ge 1$ is given by 
$$
c_1(T_{X_d})=c_1\big(\Ob_{X_d}(n+2-d)\big).
$$
Therefore, $X_d \subset \PP^{n+1}$ is Fano if and only if $d\le n+1$.
One can easily check  that if $d\le n$, then $X_d$ is covered by lines, and if $d=n+1$,  then $X_d$ is covered by conics.  In contrast, if $d\ge n+2$, then through a general point of $X_d$ there are no rational curves. 
Hence, for smooth hypersurfaces in projective spaces: 
\begin{center}
$X_d \subset \PP^{n+1}$ is Fano \ $\iff$ \ $d\le n+1$ \ $\iff$ \ $X_d$ is covered by rational curves. 
\end{center}

In the landmark paper \cite{mori79}, Mori showed that any Fano manifold is covered by rational curves. Since then, rational curves have become a fundamental tool in the study of Fano manifolds. 
Later, it was proved in  \cite{KMM2} and \cite{Campana} that a much stronger property holds: 
any Fano manifold $X$ is {\em{rationally connected}}, i.e., there are rational curves connecting any two points of $X$. 

\medskip

Fano hypersurfaces also satisfy special arithmetic properties, as illustrated by the classical Tsen's theorem. It states that a hypersurface $\mathcal{X}\subset\PP_K^{n+1}$ of degree $d\leq n+1$ over the function field $K$ of a curve always has a $K$-point. In geometric terms, Tsen's theorem says that families of hypersurfaces of degree $d\leq n+1$ in $\PP^{n+1}$ over one dimensional bases always admit holomorphic sections. 
This result has been greatly generalized by Graber, Harris and Starr in \cite{GHS}. They showed that 
proper families of rationally connected varieties over curves always admit holomorphic sections. 

In \cite{lang}, Lang provided a version of Tsen's theorem for function fields of higher dimensional varieties. 
 
\begin{thm1}[Tsen-Lang Theorem] 
Let $K$ be the function field of a variety of dimension $r\ge 1$.
If $\X\subset \PP_K^{n+1}$ is a hypersurface of degree $d$ with $d^r\leq n+1$, then $\X$ has a $K$-point.
\end{thm1}

As before, this statement can be interpreted as saying that families of hypersurfaces of degree $d$ in $\PP^{n+1}$ over $r$-dimensional bases always admit meromorphic sections provided that $d^r\leq n+1$. 
A natural problem consists in finding natural geometric conditions on general fibers of a fibration $\pi:\mathcal{X}\to B$ over a higher dimensional
base under which the conclusion on Tsen-Lang theorem holds. We note, however, that over higher dimensional bases, in addition to conditions on the fibers of the family, one must observe the existence of Brauer obstruction on the base. We do not address the latter here, and  refer to \cite{Starr_Brauer} for a discussion on the Brauer obstruction in connection with the Tsen-Lang theorem. 

\medskip

In recent years, there has been great effort towards defining suitable higher analogues of the Fano condition. Higher Fano manifolds are expected to enjoy stronger versions of several of the nice properties of Fano manifolds.

\begin{prob}\label{prob:Rr}
Find natural geometric conditions ${\mathfrak {R}_r}$ on a manifold $X$ such that:
\begin{itemize}
\item[-] for hypersurfaces $X_d \subset \PP^{n+1}$, the condition ${\mathfrak {R}_r}$ is equivalent to $d^r\leq n+1$; 
\item[-] if a complex projective manifold satisfies ${\mathfrak {R}_r}$, then $X$ is covered by rational varieties of dimension $r$; and
\item[-] families of projective manifolds satisfying ${\mathfrak {R}_r}$ over $r$-dimensional bases always admit meromorphic sections (modulo Brauer obstruction).
\end{itemize}
\end{prob}

For $r=1$, it follows from the previous discussion that the condition ${\mathfrak {R}_1}$ can be taken to be ``$X$ is rationally connected,'' or, more restrictively, ``$X$ is a Fano manifold.'' 

In a series of papers, De Jong and Starr introduce and investigate possible candidates for the condition ${\mathfrak {R}_2}$ (\cite{Starr_Hyp_RSC}, \cite{deJS_note_2fanos}, \cite{deJong_Starr_ci_RSC}, \cite{deJS_2fanos_duke}, \cite{deJHS}). They present some notions of {\em{rationally simple connectedness}}, inspired by the natural analogue of Problem~\ref{prob:Rr} in topology. Namely, let $\pi:M\to B$ be a  fibration of CW complexes with typical fiber $F$ over an $r$-dimensional base $B$. If $F$ is  $(r-1)$-connected, then $\pi$ admits a continuous section $s:B\to M$. If one draws a parallel between topology and algebraic geometry, interpreting a loop as a rational curve, one gets that the solution of the problem for $r=1$ in topology, namely $F$ is path-connected, translates precisely into the condition that the general fiber of the fibration $\pi:\mathcal{X}\to B$ is rationally connected. For $r>1$, this analogy has limitations, and De Jong and Starr's notions of rationally simple connectedness are technical and  hard to check in practice. Roughly speaking, and at the very least, one asks that a suitable irreducible component of the scheme of rational curves through two general points of $X$ is itself rationally connected. This notion led to a version of the Tsen-Lang Theorem for fibrations by rationally simply connected manifolds over surfaces in \cite{deJHS}. 
More recently, other notions of rationally simple connectedness have been investigated, for example in \cite{FGP}.

\medskip

In  \cite{deJS_note_2fanos}, De Jong and Starr introduced {\em{$2$-Fano manifolds}} as an alternative candidate to the condition ${\mathfrak {R}_2}$. 
A  projective manifold $X$ is said to be {\em $2$-Fano} if it is Fano and the second Chern character 
$\ch_2(T_X)=\frac{1}{2}c_1(T_X)^2-c_2(T_X)$ is positive, i.e., $\ch_2(T_X)\cdot S>0$ for every surface $S$ in $X$.
This condition conjecturally implies rationally simple connectedness, but it is much easier to check. 
In \cite{deJS_2fanos_duke}, it is shown that $2$-Fano manifolds satisfying some mild assumptions are covered by rational surfaces. 

In \cite{AC12}, Araujo and Castravet introduced a new approach to study $2$-Fano manifolds, via {\em {polarized minimal families of rational curves}} $(H_x, L_x)$. For a Fano manifold $X$ and a point $x\in X$, let $H_x$ be a proper irreducible component  of the scheme $\operatorname{RatCurves}^n(X,x)$, parametrizing rational curves on $X$ through $x$ having minimal anti-canonical degree.
If $x$ is a general point of $X$, then every such component is smooth and comes with a finite and birational morphism 
$\tau_x:H_x\to\PP(T_xX^{*})$, 
mapping a point parametrizing a smooth curve at $x$ to its tangent direction at $x$.
Consider the polarization $L_x:=\tau_x^*\O(1)$ on $H_x$.
Then \cite[Proposition 1.3]{AC12} gives a formula for all Chern characters of $H_x$ in terms of the
Chern characters of $X$ and $c_1(L_x)$.
As a consequence, if $X$ is $2$-Fano and $\dim(H_x)\ge 1$, then $H_x$ is Fano. 
If moreover $\text{ch}_3(X)>0$ and $\dim(H_x)\ge 2$, then $H_x$ is $2$-Fano.
This motivates the following definition. 
 
\begin{defn}
We say that a Fano manifold $X$ satisfies the condition ${\mathfrak {F}_r}$  if its Chern characters $\text{ch}_i(X)$ are positive for all $1\leq i\leq r$. This positivity condition means that $\ch_i(T_X)\cdot Z>0$ for every effective $i$-cycle $Z$ in $X$.
\end{defn}

Given a Fano manifold $X$ satisfying condition ${\mathfrak {F}_r}$, and a polarized minimal family of rational curves $(H_x, L_x)$ on it, one can ask whether, $H_x$ satisfies condition ${\mathfrak {F}_{r-1}}$.
We note the analogy with the higher notions of connectedness in topology: if a path-connected topological space $M$ is $r$-connected, then its loop spaces (with the compact-open topology) are $(r-1)$-connected.

The inductive approach introduced in  \cite{AC12} was further explored in \cite{suzuki} and \cite{Nagaoka} to show that, under some extra assumptions, Fano manifolds satisfying condition ${\mathfrak {F}_r}$ are covered by rational $r$-folds. 

Despite these recent developments, few examples of higher Fano manifolds are known. The strong restrictions on $(H_x, L_x)$ imposed by the $2$-Fano condition explain why examples are difficult to find, but they also suggest where to look for them. Following this trail,  new examples of $2$-Fano manifolds were described \cite{AC12}, including several homogeneous and quasi-homogeneous spaces. A few more examples were presented in \cite{AC13}, where Araujo and Castravet classified $2$-Fano manifolds with large index. 

\medskip

The present paper addresses the problem of finding new examples of higher Fano manifolds. 
We start by investigating rational homogeneous spaces, for which the polarized minimal families of rational curves are well described in \cite{LM03}. This search provides new examples of $2$-Fano manifolds of exceptional type $E$ and $F$ (see Section~\ref{sec:hom_background} for the notation regarding rational homogeneous spaces), and yields the following classification of 2-Fano rational homogeneous spaces of Picard rank $1$. Rational homogeneous spaces of higher Picard rank will be addressed in a forthcoming paper. 

\begin{thm} \label{thm:new examples F2}
The following is the complete list of rational homogeneous spaces of Picard rank $1$ satisfying the condition $\mathfrak{F}_2$:
\begin{enumerate}
\item[-] $A_n/P^k$, for $k=1,n$ and for $n= 2k-1,2k$ when $2 \leq k \leq \frac{n+1}{2}$;
\item[-] $B_n/P^k$, for $k=1,n$ and for $2n=3k+1$ when $2 \leq k \leq n-1$;
\item[-] $C_n/P^k$, for $k=1,n$ and for $2n=3k-2$ when $2 \leq k \leq n-1$;
\item[-] $D_n/P^k$, for $k=1,n-1,n$ and for $2n=3k+2$ when $2 \leq k < n-1$;
\item[-] $E_n/P^k$, for $n=6,7,8$ and $k = 1,2,n$;
\item[-]$F_4/P^4$;
\item[-]$G_2/P^k$, for $k=1,2$.
\end{enumerate}
\end{thm}

However, we get no new examples of Fano manifolds satisfying ${\mathfrak {F}_3}$:

\begin{thm} \label{thm:no new F3}
The only rational homogeneous spaces of Picard rank $1$ satisfying \Fthree\, are projective spaces $\PP^n$, $n\ge 3$, and quadric hypersurfaces $Q^n\subset \PP^{n+1}$, $n\ge 7$.
\end{thm}

Next we go through the list of $2$-Fano manifolds with large index in \cite{AC13} and check the ${\mathfrak {F}_3}$ condition for those. We obtain the following classification.

\begin{thm} \label{thm:3Fanos_highindex}
Let $X$ be a Fano manifold of dimension $n\ge 3$ and index $i_X\ge n-2$. If $X$ satisfies ${\mathfrak {F}_3}$, then $X$ is isomorphic to one of the following.
\begin{itemize}
\item $\PP^n$.

\item Complete intersections in projective spaces:
\begin{itemize}
\item[-]  Quadric hypersurfaces $Q^n\subset \PP^{n+1}$ with $n>6$;
\item[-]  Complete intersections of quadrics $X_{2\cdot2}\subset\PP^{n+2}$ with  $n>13$;
\item[-]  Cubic hypersurfaces $X_3\subset\PP^{n+1}$ with $n>25$;
\item[-]  Quartic hypersurfaces in $\PP^{n+1}$ with $n>62$;
\item[-]  Complete intersections $X_{2\cdot3}\subset\PP^{n+2}$ with $n>32$; 
\item[-]  Complete intersections $X_{2\cdot2\cdot2}\subset\PP^{n+3}$ with $n>20$.  
\end{itemize}

\item Complete intersections in weighted projective spaces:
\begin{itemize}
\item[-]  Degree $4$ hypersurfaces in $\PP(2,1,\ldots,1)$ with $n>55$; 
\item[-]  Degree $6$ hypersurfaces in $\PP(3,2,1,\ldots,1)$ with $n>181$; 
\item[-]  Degree $6$ hypersurfaces in $\PP(3,1,\ldots,1)$ with $n>188$; 
\item[-] Complete intersections of two quadrics in $\PP(2,1,\ldots,1)$ with $n>6$. 
\end{itemize}
\end{itemize}
\end{thm}

So the following remains an open problem. 

\begin{prob}
Find examples of Fano manifolds satisfying condition ${\mathfrak {F}_3}$ other than complete intersections in weighted projective spaces. 
\end{prob}

Our results show that conditions ${\mathfrak {F}_r}$  for $r\ge 3$ are extremely restrictive, and lead to conjectural characterizations of projective spaces and complete intersections in terms of positivity of Chern characters. It has already been asked in \cite{AC12} whether the only $n$-dimensional Fano manifold satisfying ${\mathfrak {F}_n}$ is the projective space $\PP^n$. We believe that a much stronger statement holds.

\begin{conj}
If $X$ is an $n$-dimensional Fano manifold satisfying condition ${\mathfrak {F}_k}$, with 
$k= \left \lceil{\log_2(n+1)}\right \rceil$, then $X\cong \PP^n$.
\end{conj}

\begin{prob}
For fixed $n$, find the smallest integer $k=k(n)$ such that the following holds. 
If $X$ is an $n$-dimensional Fano manifold satisfying condition ${\mathfrak {F}_{k}}$,
then $X$ is a complete intersection in a weighted projective space. 
Can this integer $k$ be chosen independently of $n$?
\end{prob}

\medskip 

Throughout this paper we work over $\C$.

The paper is organized as follows.
In Section~\ref{sec:Hx}, we review the basic theory of minimal families of rational curves on Fano manifolds, and discuss their special properties when the ambient space satisfies higher Fano conditions. 
In Section~\ref{sec:hom_background}, we provide a background on rational homogeneous spaces.
In Sections~\ref{sec:hom_classical} and \ref{sec:hom_exceptional}, we go through the classification of rational homogeneous spaces of classical and exceptional types, and check conditions ${\mathfrak {F}_{2}}$ and ${\mathfrak {F}_{3}}$. 
This case by case analysis will prove Theorems~\ref{thm:new examples F2} and \ref{thm:no new F3}.
In Section~\ref{sec:high_index}, we prove Theorem~\ref{thm:3Fanos_highindex}.

\medskip

\noindent {\bf {Acknowledgements.}} This paper anchors the invited lecture of Carolina Araujo at the Mathematical Congress of the Americas 2021. She thanks the Mathematical Council of the Americas and the organizers of the congress for the opportunity. This work grew from a working group at the  ICERM ``Women in Algebraic Geometry Collaborative Research Workshop'' in July 2020. We thank ICERM for the support and opportunity to start this research collaboration. We have also been funded by the NSF-ADVANCE Grant 150048, Career Advancement for Women through Research-focused Networks. We thank AWM for this support. Carolina Araujo was partially supported by CNPq and Faperj Research
Fellowships. Ana-Maria Castravet was partially supported by the grant ANR-20-CE40-0023. Enrica Mazzon was partially supported by the Max Planck Institute for Mathematics. The authors thank
Aravind Asok, Laurent Manivel, Mirko Mauri, and Nicolas Perrin for several enlightening discussions and comments.

%%%
\section{Minimal rational curves on higher Fano manifolds}
\label{sec:Hx}
%%%
In this section, we discuss special properties of \emph{minimal families of rational curves} on higher Fano manifolds. 
We start by reviewing the basic properties of minimal families of rational curves. We refer to \cite[Chapters I and II]{kollar} for the basic theory of rational curves on complex projective varieties. 

Let $X$ be a Fano manifold and  $x\in X$ a  point. 
There is a scheme $\rat(X,x)$ that parametrizes rational curves on $X$ through $x$.
It can be constructed as the normalization of a certain 
subscheme of the Chow variety $\chow(X)$, which parametrizes effective $1$-cycles on $X$ (see \cite[II.2.11]{kollar}).
The superscript \emph{n} in $\rat(X,x)$ stands for the normalization.
Let $H_x$ be a \emph{proper} irreducible component of $\rat(X,x)$.
We call it  a \emph{minimal family of rational curves through $x$}.
For instance, one can take $H_x$ to be an irreducible component of $\rat(X,x)$ parametrizing 
rational curves through $x$, having minimal degree with respect to $-K_X$.
If $x\in X$ is a general point, then $H_x$ is smooth.
From the universal properties of $\chow(X)$, 
we get a universal family diagram:
\begin{center}
\begin{tikzcd}
U_x \arrow[r, "\text{ev}"] \arrow[d, "\pi"']           & X ,       \\
H_x & \ %\arrow[u, "\sigma"', bend right]  & \
\end{tikzcd}
\end{center}
where $\pi$ is a $\P^1$-bundle.

The variety $H_x$ comes with a natural finite morphism  $\tau_x:  \ H_x  \to \PP(T_xX^*)$ sending  a curve that is smooth at $x$ to its tangent  direction at $x$ (see \cite[Theorems 3.3 and 3.4]{kebekus}). The morphism $\tau_x$ is also birational by \cite{hwang_mok_birationality}.
Set $L_x:=\tau_x^*\O(1)$. The pair $(H_x, L_x)$ is called a \emph{polarized minimal family of rational curves through $x$}.

In \cite{AC12}, Araujo and Castravet 
computed all the Chern characters of $H_x$ in terms of the Chern characters of $X$ and $c_1(L_x)$.
In order to state the result, first we define, for any $k\ge 1$:
\[
T:=\pi_*ev^*:N^k(X)_\R \to N^{k-1}(H_x)_\R.
\]
By  \cite[Proposition 1.3]{AC12},
\begin{equation}\label{prop:chH}
\ch_k(H_x)=\sum\limits_{j=0}^{k} A_jc_1(L_x)^j \cdot T(\ch_{k+1-j}(X))-\frac{1}{k!} c_1(L_x)^k
\end{equation}
where $A_j=\frac{(-1)^jB_j}{j!}$ and the $B_j$'s are the Bernoulli numbers.

Set $d:=\dim (H_x)$ and $\eta:=T(\ch_2(X))-\frac{L_x}{2}$. 
Then, for $1\le k\le 3$, \eqref{prop:chH} becomes:
\begin{align}
\begin{split} \label{equ:ch_1H}
c_1(H_x)
& = T(\ch_2(X)) + \frac{d}{2}c_1(L_x) 
\end{split}\\
\begin{split} \label{equ:ch_2H}
\ch_2(H_x)
& = T(\ch_3(X)) + \frac{1}{2}\bigg(c_1(H_x) - \frac{d}{2}c_1(L_x) \bigg) L_x + \frac{d-4}{12}L_x^2 
\end{split}\\
\begin{split} \label{equ:ch_3H}
\ch_3(H_x)
& =T(\ch_4(X))+\frac{1}{2}T(\ch_3(X))\cdot L_x + \frac{1}{12}T(\ch_2(X))\cdot L_x^2 - \frac{1}{6}L_x^3 \\
& =T(\ch_4(X))
+ \frac{1}{2}T(\ch_3(X))\cdot L_x 
+\frac{L_x^2}{12} \bigg(\eta - \frac{3}{2}L_x \bigg).
\end{split}
\end{align}

Using the fact that the map $T:N^k(X)_\R \to N^{k-1}(H_x)_\R$ preserves positivity for $k-1\le d$ (\cite[Lemma 2.7]{AC12}), it is possible to show an inductive structure on the class of higher Fano manifolds.

\begin{thm}
\label{prop:X4_H3Fano} \label{thm:3Fano_H2Fano}
Let $X$ be a Fano manifold, and $(H_x, L_x)$ a polarized minimal family of rational curves through a general point $x\in X$.
\begin{enumerate}
	\item ({\cite[Theorem 1.4 (2)]{AC12}}) If $X$ satisfies $\mathfrak{F}_2$ and $d\ge 1$, then $H_x$ satisfies $\mathfrak{F}_1$ (i.e., it is a Fano manifold).
	\item ({\cite[Theorem 1.4 (3)]{AC12}})
	If $X$ satisfies $\mathfrak{F}_3$ and $d\ge 2$, then $H_x$ satisfies $\mathfrak{F}_2$ and $\rho(H_x)=1$. 
	\item  If $X$ satisfies $\mathfrak{F}_4$, $d\ge 3$ and $\eta \geq \frac{3}{2} L_x$, then $H_x$ satisfies $\mathfrak{F}_3$ and $\rho(H_x)=1$. 
\end{enumerate}
\end{thm}

\begin{proof}[Proof of (3)]
Suppose that the Fano manifold $X$ satisfies $\mathfrak{F}_4$, $d\ge 3$ and $\eta \geq \frac{3}{2} L_x$. We know from (2) that $H_x$ is a Fano manifold satisfying $\mathfrak{F}_2$ and $\rho(H_x)=1$. Using the fact that the map $T:N^k(X)_\R \to N^{k-1}(H_x)_\R$ preserves positivity for $k-1\le d$, it follows from \eqref{equ:ch_3H} that $\ch_3(H_x) > 0$, and thus $H_x$ satisfies $\mathfrak{F}_3$.
\end{proof}

\begin{rem}
Let the assumptions be as in \cref{prop:X4_H3Fano}(3). By \cite[Proof of Theorem 1.4, p.99]{AC12}, the class $\eta$ is nef. Since both $2\eta$ and $2 \cdot \ch_2(X)$ are integral classes, the assumption $\eta \geq \frac{3}{2} L_x$ is satisfied, except if $\eta =0, \frac{1}{2}L_x, L_x$. 
\end{rem}

We end this section with the description of $(H_x, L_x)$ when $X$ is a Fano manifold satisfying $\mathfrak{F}_2$.
%(see \cite[Theorem 1.4 (2)]{AC12}).

\begin{thm}[{\cite[Theorem 1.4 (2)]{AC12}}]
\label{thmAC12:HxLx}
Let $X$ be a Fano manifold, $(H_x, L_x)$ a polarized minimal family of rational curves through a general point $x\in X$, and set $d=\dim H_x$. 
\begin{enumerate} 
\item If $X$ is 2-Fano, then
 $H_x$ is a Fano manifold with $\Pic(H_x)=\Z \cdot [L_x]$, except if $(H_x,L_x)$ is isomorphic to one of the following \begin{enumerate} 
	\item $\Big(\PP^m\times \PP^m, p_{_1}^*\mathcal{O}(1)\otimes p_{_2}^*\mathcal{O}(1)\Big)$, with $d=2m$,
	\item $\Big(\PP^{m+1}\times\PP^{m} \ ,\  p_{_1}^*\mathcal{O}(1)\otimes p_{_2}^*\mathcal{O}(1)\Big)$, 
		with $d=2m+1$,
	\item $\Big(\PP_{\PP^{m+1}}\Big(\mathcal{O}(2)\oplus \mathcal{O}(1)^{^{\oplus m}}\Big) \ , \ \mathcal{O}_{\PP}(1)\Big)$, 
		with $d=2m+1$,
	\item  $\Big(\PP^{m}\times Q^{m+1} \ ,\  p_{_1}^*\mathcal{O}(1)\otimes p_{_2}^*\mathcal{O}(1)\Big)$, 
		with $d=2m+1$
	\item  $\Big(\PP_{\PP^{m+1}}\big(T_{\PP^{m+1}}\big) \ , \ \mathcal{O}_{\PP}(1)\Big)$, with $d=2m+1$.
	\item  $\Big(\PP^d,\mathcal{O}(2)\Big)$, or
	\item  $\Big(\PP^1,\mathcal{O}(3)\Big)$.
\end{enumerate}
\item Suppose $b_4(X)=1$. Then $X$ is 2-Fano if and only if $-2K_{H_x}-dL_x$ is ample.

\end{enumerate}
\end{thm}

%%%
\section{Background on homogeneous spaces}
\label{sec:hom_background}

In order to investigate the conditions $\mathfrak{F}_2$ and \Fthree\, on rational homogeneous spaces, we start with a brief recollection of definitions and properties from the theory of homogeneous varieties. 
We refer to \cite{Bourbaki, FH, Humphreys, Milne, Springer} for the basics on algebraic groups and Lie algebras. 

\medskip

An algebraic variety is called \emph{homogeneous} if there is an algebraic group acting transitively on it.
A classical theorem by Borel-Remmert (see {\cite[page 101]{Ak}}) % for an English reference) 
asserts that if $X$ is a projective homogeneous variety, then it is isomorphic to a product
$$ X \cong \textrm{Alb}(X) \times X',$$
where $\textrm{Alb}(X)$ is the Albanese variety of $X$ and $X'$ is a projective rational homogeneous space. The latter can in turn be written as a product
$$ X' \cong G_1/P_1 \times \cdots \times G_l/P_l
$$
for some simple algebraic groups $G_i$ and parabolic subgroups $P_i$. 
%Moreover, the list of homogeneous spaces under $G$ will only depend on its Lie algebra $\gfr$, which is uniquely encoded by a Dynkin diagram, so for every Dynkin diagram, it is enough to study only one group.
%
%Note that parabolic groups are of importance to algebraic geometry because quotients $G_i/P_i$ are smooth proper varieties and they are Fano manifolds. 
%Connected, simply connected, simple algebraic groups are uniquely defined by their Dynkin diagram, and the rational homogeneous spaces are encoded in marking on the node of the diagrams. We explain in more detail these correspondences.
%
Next we explain how to encode the properties of $G_i/P_i$ in a marked Dynkin diagram.

%%%
\subsection{Semisimple algebraic groups and Dynkin diagrams}
\label{subsec:background:gps}

Given a semisimple algebraic group $G$, we can choose a Borel subgroup $B \subset G$ (a maximal closed and connected solvable algebraic subgroup) and a maximal torus $T \subset B \subset G$.
For example, when $G = \mathrm{SL}_n$ we can choose $B$ to be the subgroup of lower-triangular matrices and $T\cong \Gm^{n-1}$ to be the group of all diagonal matrices.
%
%one associates the root system of the Lie algebra $\mathfrak{g}$ of $G$ relative to $\mathfrak{t}$, Lie algebra of $T$. Then the simple roots of the root system correspond to nodes of the associated Dynkin diagram, and the number of connecting lines depend on the angle between the two simple roots. In the case of a multiple connection, we also draw an arrow that points in the direction of the shorter root.
%
Define the \emph{character lattice} 
$$ X^*(T) \ = \ \Hom(T,\Gm) \ \cong \ \Z^l
.$$
Let $\gfr$ and $\tfr$ denote the Lie algebras of the groups $G$ and $T$, respectively.
Base changing $X^*(T)$ to the complex numbers yields a complex vector space that is naturally isomorphic to the dual $\tfr^\vee$, so we can consider the character lattice as a subset
$$ X^*(T) \subset \tfr^\vee
.$$
By considering $\gfr$ as a $T$-representation under the adjoint action, we can write its weight decomposition. Moreover, factoring out the subrepresentation $\tfr$ whose weights are zero, we define a finite subset $\Phi \subset X^*(T)\setminus\{0\}$ of \emph{roots}, so that we have the following form of the weight decomposition:
$$ \gfr = \tfr \oplus \bigoplus_{\alpha \in \Phi} \gfr_{\alpha}
.$$
For each $\alpha$, the weight space $\gfr_\alpha$ is 1-dimensional.  Moreover, the set $\Phi$ forms a so-called \emph{root system}. 
Taking a generic hyperplane 
$$ H \subset X^*(T)_\R
,$$
we can divide $X^*(T)_\R$ into two half-spaces: positive and negative. The roots in the positive half-space will be called \emph{positive roots}, the set of which is denoted by $\Phi^+$. Similarly we define the set of \emph{negative roots} $\Phi^-$. Inside $\Phi^+$, we can choose a set of \emph{simple roots}
$$\{\alpha_1,\dots ,\alpha_l\} \subset \Phi^+
$$
that satisfy the following conditions: first, they form an $\R$-basis for $X^*(T)_\R$; and second, none of them is a nonnegative linear combination of other positive roots. It turns out that these conditions, together with the axioms of root systems, imply that the possible angles between the $\alpha_i$ (with respect to a scalar product invariant under the normalizer of the fixed maximal torus) are $90^\circ$, $120^\circ$, $135^\circ$, $150^\circ$. If we mark points on a plane labeled by the simple roots, we can connect them to each other with a number of lines that depends on the angle. For the angles listed above, we draw 0, 1, 2 and 3 lines, respectively. Moreover, for double and triple bonds we draw an arrow from the node marking the longer simple root to the shorter simple root. Nodes connected with just one line correspond to simple roots of the same length.

This procedure results in a \emph{Dynkin diagram} corresponding to the root system. 
For example, for $G=\mathrm{SL}_3$, the torus is of rank two, the root system has six vectors pointing at vertices of a regular hexagon, and the two simple roots have angle $120^\circ$, so the corresponding Dynkin diagram is $A_2$:
$$ \mathrm{SL}_3 \quad
\begin{dynkinDiagram}[mark=o]A2
\end{dynkinDiagram} \quad A_2.
$$

\noindent We recall all possible connected Dynkin diagrams one can obtain; our convention is to use the ordering of roots as in Bourbaki \cite{Bourbaki}.

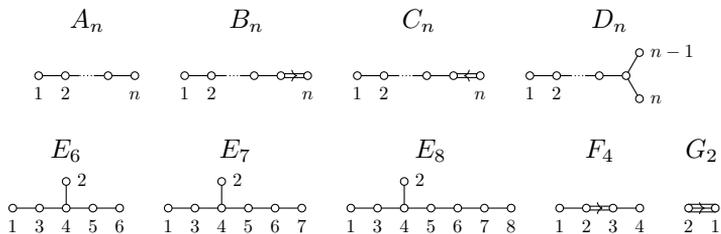
\begin{figure}[h]
\begin{center}
\begin{tabular}{cccc}
$A_n$ & $B_n$ & $C_n$ & $D_n$ \\
\begin{dynkinDiagram}[mark=o,labels={1,2,,n}]A{}
\end{dynkinDiagram} 
&  \begin{dynkinDiagram}[mark=o,labels={1,2,,,n}]B{}
%\dynkinRootMark{*}5
\end{dynkinDiagram} 
&  \begin{dynkinDiagram}[mark=o,labels={1,2,,,n}]C{}
%\dynkinRootMark{*}1
%\dynkinRootMark{*}2
%\dynkinRootMark{*}3
%\dynkinRootMark{*}4
\end{dynkinDiagram} 
& \begin{dynkinDiagram}[mark=o,labels={1,2,,,n-1,n}]D{}
\end{dynkinDiagram}
\end{tabular}
\end{center}
\begin{center}
\vspace{5pt}
\begin{tabular}{ccccc}
$E_6$ & $E_7$ & $E_8$ & $F_4$ & $G_2$ \\
\begin{dynkinDiagram}[mark=o,label]E6
\end{dynkinDiagram}
& \begin{dynkinDiagram}[mark=o,label]E7
\end{dynkinDiagram}
& \begin{dynkinDiagram}[mark=o,label]E8
\end{dynkinDiagram}
& \begin{dynkinDiagram}[mark=o,label]F4
%\dynkinRootMark{*}3
%\dynkinRootMark{*}4
\end{dynkinDiagram}
& \begin{dynkinDiagram}[mark=o,label,ordering=Carter]G2
%\dynkinRootMark{*}2
\end{dynkinDiagram}
\end{tabular}
\end{center}
%\caption{Dynkin diagrams (ordering of roots as in Bourbaki)}
\end{figure}

\subsection{Parabolic subgroups}\label{par subgr}
Let $B \subset G$ be a fixed Borel subgroup as above; we assume that its Lie algebra has nonpositive weights.
A \emph{parabolic subgroup} $P$ of $G$ is a connected subgroup that contains some Borel subgroup; up to conjugation, we may assume that $P$ contains $B$. For example, $B$ and $G$ are trivially parabolic subgroups. 
To a fixed set of nodes in the Dynkin diagram, we can assign a parabolic subgroup of $G$ as follows. Choose a subset $I \subset \{1,\dots,l\}$ of labels. Let $\Phi^+_I$ denote the set of positive roots generated (under taking nonnegative linear combinations) by 
$$ \{ \alpha_i \mid i\in I
\} .$$
Consider the following subspace $\pfr_I \subset \gfr$, which one can verify forms a Lie subalgebra: 
$$ \pfr_I = \tfr \oplus
\bigoplus_{\alpha \in \Phi^-} \gfr_\alpha
\bigoplus_{\alpha \in \Phi^+_I} \gfr_\alpha
.$$
Then the parabolic group $P_I$ corresponding to $I$ will be the subgroup of $G$ that contains $B$ and whose tangent space at the identity element is $\pfr_I \subset \gfr$.

The assignment of $P_I$ to each $I$ classifies parabolic subgroups of $G$ up to conjugation, i.e., $G$-equivalence classes of parabolic subgroups of $G$ are in bijection with subsets of nodes in the Dynkin diagram.
If the subset $I$ is empty, then $P_I = B$; if $I = \{1,\dots ,l\}$, then $P_I = G$. If the complement $\{1,\dots ,l\} \setminus I = \{k\}$ is of cardinality one, we write $P_I=P^k$ and this is a \emph{maximal parabolic} subgroup, i.e. maximal among those not equal to the whole $G$.

\subsection{Rational homogeneous spaces}\label{rhs}
A \emph{projective rational homogeneous space} is a quotient of a semisimple algebraic group by a parabolic subgroup.
The quotient is a rational projective variety, and its geometry is heavily controlled by combinatorics. For instance, by \cref{par subgr} a rational homogeneous space $G/P$ corresponds to the datum of a marked Dynkin diagram (possibly not connected). 

The spaces $G/P_I$ are Fano manifolds of dimension
$$ \dim G/P_I = \dim G - \dim P_I = | \Phi^+ \setminus \Phi^+_I |,
$$
computed by subtracting dimensions of tangent spaces at the identity element, and of Picard rank $l-|I|$ (see for example \cite[Prop. 1.20]{Pasquier}, \cite[Prop. 6.5, Thm. 9.5]{Snow}). In particular, when $G$ is simple, the quotients by maximal parabolic subgroups $X=G/P^k$ are Fano varieties of Picard rank one. In this case, the ample generator $\mathcal{O}_X(1)$ of the Picard group gives the smallest $G$-equivariant embedding of $X$ in a projective space. 

%$G/P_k\subseteq \mathbb P(V_{\varpi_k})$, where $V_{\varpi_k}$ is the irreducible representation associated to the {\em fundamental dominant weight} $\varpi_k$.  

\medskip

%%%
%\subsection{Weyl group.}
%%%
%Weyl group plays a large role in representation theory, and importantly for us, the geometry of homogeneous spaces, namely description of Schubert cells which give a certain cell decomposition of a rational homogeneous space. We will now recall the main background that we will use.

The \emph{Weyl group} $W = N_G(T) / Z_G(T)$ is defined as the quotient of the normalizer of the fixed maximal torus $T \subset G$ by its centralizer. It is a finite group that acts faithfully on $X^*(T)$ and permutes the roots $\Phi$, see \cite[\S 7.1.4]{Springer}.
As in Section \ref{subsec:background:gps}, fix an $N_G(T)$-invariant scalar product $(\,,\,)$ on $X^*(T)_\R$ -- it will naturally be Weyl group invariant for the induced action of $W$. For a root $\alpha \in \Phi$, one can define a reflection:
$$ s_\alpha(x) = x - 
2 \cdot \frac{(x,\alpha)}{(\alpha,\alpha)} \cdot \alpha 
.$$
The Weyl group $W$ is generated by $s_i$ --
reflections with respect to simple roots $\alpha_i$. The length $l(w)$ of an element $w\in W$ is the smallest integer $r\geq0$ such that $w$ is a product of $r$ simple reflections
$s_{i}$; a reduced decomposition of $w$ is a sequence $(s_{{i_1}},\ldots,s_{{i_r}})$ 
such that  
$w=s_{i_1}\cdot\ldots\cdot s_{i_r}$
and with $r=l(w)$.

\medskip

Let $I\subset\{1,\ldots,l\}$ and let $P=P_I$. 
We denote by $W_P$ the subgroup of $W$ generated by reflections $s_{i}$, for all the simple roots $\alpha_i$ with $i\in I$. 
In each coset of $W/W_P$, there exists a unique representative $w$ of minimal length 
\cite[Proposition 5.1(iii)]{BGG}.
We denote by $W^P\subseteq W$ the set of such representatives. We can identify the elements of $W^P$ as the elements $w\in W$ for which any reduced decomposition $(s_{{i_1}},\ldots, s_{{i_r}})$ of $w$ has $i_r\notin I$ (see \cite[Proposition 5.1(iii)]{BGG} and \cite[p. 50, last Corollary]{Humphreys_72}). 
In particular, there is a bijection between $W^P$ and the quotient group $W/W_P$. 

The Chow group $A^*(G/P)$ is generated by the classes of {\em Schubert subvarieties} $X(w)$, for all $w\in W^P$; the classes of Schubert subvarieties form an additive basis for
$A^*(G/P)\cong H^*(G/P;\mathbb Z)$ (see for example \cite[Section 2.2]{CMP}, \cite[Section 3]{Snow}). 
For all $j\geq0$, the Betti number $b_{2j}(G/P)$ equals the number of elements in $W^P$ of length $j$ since $\dim X(w) = 2l(w)$, see \cite[Proposition 5.1]{BGG}. 
% ENGLISH VERSION OF BGG:
% http://www.math.tau.ac.il/~bernstei/Publication_list/publication_texts/BGG-SchubCells-Usp.pdf

As an application of the above discussion, we get the following lemma:

\begin{lem} \label{lem:Betti numbers}
When $G$ is simple and $P=P^k$, we have $b_2(G/P^k)=1$ and $b_4(G/P^k)$ equals the number of simple roots adjacent to $\alpha_k$.
\end{lem}
\begin{proof}
Recall that $P^k = P_I$ for $I = \{1,\dots,l\} \setminus \{k\}$. From the above description of $W^P$, the only length one element $w \in W^P$ is $w = s_k$, hence
$b_2(G/P^k) = 1$.
For $b_4(G/P^k)$, 
length two elements $w\in W^P$ are of the form $w = s_i \cdot s_k$ with $i\neq k$.
Using the fact that $s_\alpha\cdot s_\beta=s_\beta\cdot s_\alpha$ for non-adjacent (hence orthogonal) simple roots $\alpha$ and $\beta$,
we can see that for non-adjacent $i$ and $k$ the reduced decomposition $w = s_i \cdot s_k$ cannot be of minimal length in $[w] \in W/W_P$, because $[w] = [s_k \cdot s_i] = [s_k]$. 
So $b_4(G/P^k)$ is the number of simple roots adjacent to $\alpha_k$. 
\end{proof}

For $G$ simple and a parabolic subgroup $P=P^k$ corresponding to a non-short root $\alpha_k$, i.e., a root such that no arrows in the diagram point in the direction of the corresponding node, Landsberg and Manivel describe the minimal family of rational curves $H_x$ through a point of $X= G/P^k$:

\begin{thm}[{\cite[Theorem 4.8]{LM03}, see also \cite[Theorem 2.5]{LM04} and the subsequent paragraph}] \label{thm:LM_Hx}
Let $X=G/P^k$ be a rational homogeneous variety such that $\alpha_k$ is not short. Then $H_x$ is homogeneous and the associated marked diagram is determined as follows: remove the node corresponding to $k$ and mark the nodes adjacent to $k$. 
Moreover, the embedding of $H_x$ in $\PP(T_x X)$ is minimal if and only if the Dynkin diagram of $G$ is simply laced, i.e., without multiple edges.
\end{thm}
\smallskip

\begin{ex}
    \begin{tabular}{cc}
         $X=E_8/P^6$ 
         & $H_x=\textrm{Seg}(\OGr_+(5,10) \times \PP^2)$  \\
         \begin{dynkinDiagram}[mark=o,label]E8 \dynkinRootMark{x}6 \end{dynkinDiagram} 
         & \begin{dynkinDiagram}[mark=o,label]D5 \dynkinRootMark{x}5 \end{dynkinDiagram} 
         \begin{dynkinDiagram}[mark=o,label]A2 \dynkinRootMark{x}1 \end{dynkinDiagram} 
    \end{tabular}
\end{ex}

\medskip
In Sections \ref{sec:hom_classical} and \ref{sec:hom_exceptional}, we study rational homogeneous spaces of Picard rank one.
%which are quotients $G/P^k$ by a maximal parabolic subgroup. 
We do a type by type analysis, which yields a proof of our main results \cref{thm:new examples F2} and \cref{thm:no new F3}.

\section{Rational homogeneous varieties of classical type}
\label{sec:hom_classical}

In this section we consider rational homogeneous spaces constructed from Dynkin diagrams of type $A,B,C$ and $D$ as quotients by a maximal parabolic subgroup $P^k$. These give rise to the following varieties, see \cite[\S 1.1]{Man20} for a reference:

\begin{center}
\renewcommand{\arraystretch}{1.2}
\begin{tabular}{cc|ll}
$A_n$ & 
\begin{dynkinDiagram}[mark=o,labels={1,2,,n}]A{}
\end{dynkinDiagram} 
&
& $A_n/P^k=\Gr(k,n+1)$ 
\\ \hline
$B_n$ & \begin{dynkinDiagram}[mark=o,labels={1,2,,,n}]B{}
%\dynkinRootMark{*}5
\end{dynkinDiagram} 
& for $k<n$:
& $B_n/P^k=\OGr(k,2n+1)$
\\
& & for $k=n$:
& $B_n/P^n=\OGr_+(n+1,2(n+1))$ 
 \\ \hline
$C_n$ & \begin{dynkinDiagram}[mark=o,labels={1,2,,,n}]C{}
%\dynkinRootMark{*}1
%\dynkinRootMark{*}2
%\dynkinRootMark{*}3
%\dynkinRootMark{*}4
\end{dynkinDiagram} 
&  & $C_n/P^k=\SGr(k,2n)$ 
 \\ \hline
$D_n$ & \begin{dynkinDiagram}[mark=o,labels={1,2,,,n-1,n}]D{}
\end{dynkinDiagram}
& for $k<n-1$:
& $D_n/P^k=\OGr(k,2n)$ 
 \\
&  & for $k\in \{n-1,n\}$: 
& $D_n/P^{n-1} \cong D_n/P^n=\OGr_+(n,2n)$ 
\end{tabular}
\end{center}

\medskip

We explain the notation in the table. We denote by $\Gr(k,n)$ the Grassmannian of $k$-dimensional linear subspaces of an $n$-dimensional vector space $V$. When $n \neq 2k$, $\OGr(k,n)$ is the subvariety of the Grassmannian $\Gr(k,n)$ parametrizing linear subspaces of $V$ that are isotropic with respect to a non-degenerate symmetric bilinear form $Q$. If $n=2k$, $\OGr(k,n)$ has two isomorphic disjoint components $\OGr_\pm(k,2k)$ such that $\OGr_+(k,2k) \cong \OGr_-(k,2k) \cong \OGr(k-1,2k-1).$ Finally, observe that for $n=2k+1$ we have $\OGr(k,2k+1) \cong \OGr_+(k+1,2(k+1))$.
We denote by $\SGr(k,n)$ the subvariety of the Grassmannian $\Gr(k,n)$ parameterizing linear subspaces of $V$ that are isotropic with respect to a non-degenerate antisymmetric bilinear form $\omega$.
Whenever we refer to $\SGr(k,n)$ we will assume $n$ is even. 
 In all cases except the case of the spinor variety $\OGr_+(k,2k)$ (i.e., $B_{k-1}/P^{k-1}$,  $D_k/P^{k-1}$, $D_k/P^{k}$), the ample generator $\mathcal{O}_X(1)$ of the Picard group of $X=G/P$ is the restriction of the corresponding Pl\"ucker embedding polarization.

\begin{rem}
The integral cohomology of $\Gr(k,n)$ admits a basis of Schubert cycles indexed by Young diagrams that fit into a rectangle $k\times (n-k)$. Given a diagram $\lambda$, the codimension of the cycle $\sigma_\lambda$ corresponding to $\lambda$ is equal to the number of boxes. The intersection of two Schubert cycles is calculated by the Littlewood-Richardson rule, and for some special cycles this is also known as the Pieri rule.
For further facts and proofs, we refer to Fulton's classical textbook \cite[\S9.4]{Fulton}.
\end{rem}

We recall in a table, results from \cite[\S 5]{AC12} and \cite[\S 6.2]{AC13} on the condition $\mathfrak{F}_2$ for homogeneous spaces of classical type. 
%In particular, they determine when the condition $\mathfrak{F}_2$ is satisfied. 
Note that we have $\Gr(1,n)\cong\PP^{n-1}$, $\SGr(1,n) \cong \PP^{n-1}$ ($n$ even) and  $\OGr(1,n)=Q^{n-2}$. 

\begin{center}
\renewcommand{\arraystretch}{1.25}
\begin{tabular}{ll|l|l}
& & $\dim$ 
& $\mathfrak{F}_2$ is satisfied  $\iff$
\\ \hline
for $2 \leq k \leq \frac{n}{2}$
& $\Gr(k,n)$ 
& $k(n-k)$ 
& $n\in\{2k, 2k+1\}$
\\ \hline
for $2 \leq k < \frac{n}{2}-1$
& $\OGr(k,n)$ 
& $\frac{k(2n-3k-1)}{2}$
& $n=3k+2$
\\ \hline
& $\OGr_+(k,2k)$ 
& $\frac{k(k-1)}{2}$
& $\forall k$
\\ \hline
for $2 \leq k \leq \frac{n}{2}$
& $\SGr(k,n)$ 
& $\frac{k(2n-3k+1)}{2}$
& $n=3k-2$
\\  \hline
& $\SGr(k,2k)$ 
& $\frac{k(k+1)}{2}$
& $\forall k$
\end{tabular}
\end{center}
We continue the previous table, collecting the results on $(H_x,L_x)$ from \cite[\S 5]{AC12} and \cref{thm:LM_Hx}:
%\cite[\S 3.3]{BK21} for $\OGr_+(k,2k)$:
\begin{center}
\renewcommand{\arraystretch}{1.25}
\begin{tabular}{ll|l}
&
& $(H_x,L_x)$
\\ \hline
for $2 \leq k \leq \frac{n}{2}$ & 
$\Gr(k,n)$ 
& $\Big(\PP^{k-1}\times \PP^{n-k-1}, p_{_1}^*\mathcal{O}(1)\otimes p_{_2}^*\mathcal{O}(1)\Big)$
\\ \hline
for $2 \leq k < \frac{n}{2}-1$ & 
$\OGr(k,n)$ 
& $\Big(\PP^{k-1}\times Q^{n-2k-2} \ ,\  p_{_1}^*\mathcal{O}(1)\otimes p_{_2}^*\mathcal{O}(1)\Big)$
\\ \hline
&
$\OGr_+(k,2k)$ 
& $(\Gr(2,k), H)$
\\ \hline
for $2 \leq k \leq \frac{n}{2}$ & 
$\SGr(k,n)$ 
& $\Big(\PP_{\PP^{k-1}}\Big(\mathcal{O}(2)\oplus \mathcal{O}(1)^{^{n-2k}}\Big) \ , \ \mathcal{O}_{\PP}(1)\Big)$
\\  \hline
& 
$\SGr(k,2k)$ 
& $(\PP^{k-1}, \mathcal{O}(2))$
\end{tabular}
\end{center}

\noindent  where $H$ is a hyperplane class under the Pl{\"u}cker embedding.

%%%
\subsection{Type A: Grassmannians}
%%%

%\begin{thm}[{\cite[Theorem 1.4]{AC12}}]
%\label{thm:3Fano_H2Fano}
%Let $X$ be a Fano manifold and $H_x$ be the space of minimal rational curves through a general point $x\in X$. 
%If $\ch_2(X) > 0$, $\ch_3(X) \geq 0$ and $\dim(H_x) \geq 2$, then $H_x$ is a Fano manifold with $\rho(H_x)=1$ and $\ch_2(H_x) > 0$.
%\end{thm}

%The group $\SL_n$ has Dynkin diagram $A_{n-1}$, and the homogeneous space looks familiar: $\SL_n / P^k \cong \mathrm{Gr}(k,n)$. When $k=1$ or $k=n-1$, this homogeneous space is isomorphic to $\PS^{n-1}$ which trivially satisfies $\mathfrak{F}_n$, but all other Grassmannians do not even satisfy \Fthree.

\begin{prop} \label{prop:F3_A}
For $2 \leq k \leq n-2$, $\Gr(k,n)$ does not satisfy the condition \Fthree.
\end{prop}

\begin{proof}
Let $X=\Gr(k,n)$.
As $H_x\cong \PS^{k-1} \times \PS^{n-k-1}$
is a product, by \cite[\S 3.3]{deJS_note_2fanos} it does not satisfy $\mathfrak{F}_2$. By \cref{thm:3Fano_H2Fano} it follows that $X$ does not satisfy \Fthree.
\end{proof}

%When $X=\Gr(2,5)$, we have that $H_x=\P^1 \times \P^2$ by \cite[\S 1.4.4]{hwang_ICTP}; this is neither $\mathfrak{F}_2$, nor has Picard number $1$, therefore by \cref{thm:3Fano_H2Fano}, $\Gr(2,5)$ is not $\mathfrak{F}_3$.

\subsection{Types B and D: orthogonal Grassmannians}

%Before we state our result for orthogonal grassmannians, let us summarize the notation that we will use in the proof.

\begin{prop}
For $2 \leq k < \frac{n}{2}-1$, $\OGr(k,n)$ does not satisfy the condition \Fthree.
\end{prop}

\begin{proof}
Let $X=\OGr(k,n)$. As $H_x$ is a product, by \cite[\S 3.3]{deJS_note_2fanos} it does not satisfy $\mathfrak{F}_2$. By \cref{thm:3Fano_H2Fano} it follows that $X$ does not satisfy \Fthree.
\end{proof}

Now we consider $\OGr_+(k,2k)$, whose family of minimal rational curves through a general point is $H_x \cong \Gr(2,k)$, by \cref{thm:LM_Hx}. We have that $\Gr(2,k)$ satisfies $\mathfrak{F}_2$ if and only if $4 \leq k \leq 5$ by \cite[\S 5.2]{AC12}. Thus, by \cref{thm:3Fano_H2Fano} the condition $k \in \{4,5\}$ is necessary for $\OGr_+(k,2k)$ to satisfy the condition \Fthree. However, we show that $\OGr_+(k,2k)$ does not satisfy \Fthree\, by considering a $3$-cycle whose intersection with $\ch_3(\OGr_+(k,2k))$ is non-positive. 
To this end, we start by recalling a result by Coskun about restriction of the Schubert cycles from $\Gr(k,n)$ to $\OGr_+(k,n)$.

\begin{prop}[{\cite[Proposition 6.2]{izzet}}]
\label{prop:Coskun_OG}
Let $j:\OGr(k,n) \hookrightarrow \Gr(k,n)$ be the natural inclusion, and $\sigma_{\lambda_1,\dots,\lambda_k}$ a Schubert cycle in $\Gr(k,n)$. Then
\begin{itemize}
    \item[(1)] $j^*\sigma_{\lambda_1,\dots,\lambda_k}=0$ unless $n-k-i \ge \lambda_i$ for all $i$ with $1 \le i \le k$.
    \item[(2)] 
    Suppose that $n-k-i > \lambda_i$ for all $1 \le i \le k$. Then $j^*\sigma_{\lambda_1,\dots,\lambda_k}$ is effective and nonzero.
    \item[(3)] 
    Suppose that $n=2k$ and $k-i = \lambda_i$ for all $1 \le i \le k$. Then $j^*\sigma_{\lambda_1,\dots,\lambda_k}$ is $2^{k-1}$ times the Poincar\'{e} dual of a point.
\end{itemize}
\end{prop}

We recall from \cite[\S 6.2]{AC13} the expression of the third Chern character in terms of restrictions of Schubert cycles from $\Gr(k,2k)$:
    $$
    \ch_3(\OGr_+(k,2k))
    =-\frac{k+7}{6}j^*\sigma_3+\frac{k+4}{6} 
    j^*\sigma_{2,1}
    - \frac{k+1}{6}j^*\sigma_{1,1,1}.
    $$

%\begin{itemize}
   % \item[-] If $n \neq 2k$, $\OGr(k,n)$ is a Fano manifold of dimension $\frac{k(2n-3k-1)}{2}$ and Picard number $1$. If $n=2k$, $\OGr(k,n)$ has two isomorphic disjoint components $\OGr_\pm(k,2k)$ such that  $$\OGr_+(k,2k) \cong \OGr_-(k,2k) \cong \OGr(k-1,2k-1),$$ hence they are of dimension $\frac{k(k-1)}{2}$.
   % \item[-] $\OGr(k,n)$ is the zero locus in $\Gr(k,n)$ of a global section of the vector bundle $\Sym^2(\mathcal{S}^*)$, where $\mathcal S^*$ is the universal quotient bundle on $\mathrm{Gr}(k,n)$ \cite[\S5.4]{AC12}.
    %\item[-] When $n=2k$, the family of minimal rational curves on $\OGr_+(k,2k)$ is $H_x \cong \Gr(2,k)$, by \cite[\S 3.3]{BK21}. Moreover, by \cite[\S 5.2]{AC12}, $\Gr(2,k)$ is $\mathfrak{F}_2$ if and only if $4 \leq k \leq 5$. It follows from \cref{thm:3Fano_H2Fano} that the condition $k=4,5$ is necessary for $\OGr_+(k,2k)$ being $\mathfrak{F}_3$.
    %\\When $n=2k+1$, observe that $\OGr(k,2k+1) \cong \OGr_+(k+1,2(k+1))$.
    %\\When $n > 2k +2$,  $(H_x,L_x)\cong \big(\PP^{k-1}\times Q^{n-2k-2}, p_{_1}^*\mathcal{O}(1)\otimes p_{_2}^*\mathcal{O}(1)\big)$ by \cite[\S 5.4]{AC12}; it follows again from \cref{thm:3Fano_H2Fano} that $\OGr(k,n)$ is not $\mathfrak{F}_3$ for $n > 2k +2$.
    %\item[-] By \cite[\S 6.2]{AC13}, 
    %$$\ch_3(\OGr(k,n))=\frac{n-3k-7}{6}\sigma_3-\frac{n-3k-4}{6}\sigma_{2,1}+ \frac{n-3k-1}{6}\sigma_{1,1,1}.$$
%\end{itemize}
    
\begin{lem} \label{lem:OG_relation}
For $X=\OGr_+(k,2k)$ with $k\ge 3$, we have $j^*\sigma_3=j^*\sigma_{1,1,1}$ and
$$\ch_3(X)= \frac{k+4}{2} \bigg(\frac{1}{6}j^*\sigma_1^3 - j^*\sigma_3 \bigg).$$
\end{lem}

\begin{proof}
By~\cite[Claim 33]{AC13} the equality $j^*\sigma_2=j^*\sigma_{1,1}=\frac{1}{2}j^*\sigma_1^2$ holds on $X$. Applying Pieri's formula we have
\begin{align*}
\frac{1}{2}j^*\sigma_1^3 
& = j^*(\sigma_1 \cdot \sigma_2) 
\stackrel{\text{Pieri}}{=} j^*\sigma_{2,1}+j^*\sigma_3
\\
\frac{1}{2}j^*\sigma_1^3
& = j^*(\sigma_1\cdot \sigma_{1,1})
\stackrel{\text{Pieri}}{=} j^*\sigma_{2,1}+j^*\sigma_{1,1,1}.
\end{align*}
This implies that $j^*\sigma_3=j^*\sigma_{1,1,1}$ and
$$
\ch_3(X)
=-\frac{k+7}{6}j^*\sigma_3+\frac{k+4}{6} 
\bigg( \frac{1}{2} j^*\sigma_1^3 - j^*\sigma_3 \bigg)
- \frac{k+1}{6} j^*\sigma_{3} 
= \frac{k+4}{2}  \bigg( \frac{1}{6} j^*\sigma_1^3 - j^*\sigma_3 \bigg).
$$
\end{proof}

\begin{prop} \label{prop:F3_OG+}
For $k=4$ or $5$, $\OGr_+(k,2k)$ does not satisfy the condition \Fthree.
\end{prop}

\begin{proof}
Consider the codimension $3$ Schubert cycle $\sigma_{2,1}$ on $\Gr(4,8)$, and respectively the codimension $7$ Schubert cycle $\sigma_{3,2,1,1}$ on $\Gr(5,10)$. We have
\begin{align*}
\frac{6}{8}\ch_3(\OGr_+(4,8))\cdot j^*\sigma_{2,1}
 \stackrel{\text{\cref{lem:OG_relation}}}{=} 
\quad & \quad 
j^* \bigg( \frac{1}{2}\sigma_1^3 \cdot \sigma_{2,1} -3\,\sigma_3 \cdot \sigma_{2,1} \bigg)
\\
\stackrel{\substack{\text{Pieri } +\\ \text{\cref{prop:Coskun_OG}}}}{=} 
& \quad 
j^* \bigg( \frac{1}{2} 6 \, \sigma_{3,2,1} - 3 \,\sigma_{3,2,1} \bigg) = 0
\\
\frac{6}{9}\ch_3(\OGr_+(5,10))\cdot j^*\sigma_{3,2,1,1}
 \stackrel{\text{\cref{lem:OG_relation}}}{=} 
\quad & \quad 
j^* \bigg( \frac{1}{2}\sigma_1^3 \cdot \sigma_{3,2,1,1} -3\,\sigma_3 \cdot \sigma_{3,2,1,1} \bigg)
\\
\stackrel{\substack{\text{Pieri } +\\ \text{\cref{prop:Coskun_OG}}}}{=} 
& \quad 
j^* \bigg( \frac{1}{2} 6 \, \sigma_{4,3,2,1} - 3 \, \sigma_{4,3,2,1} \bigg) = 0.
\end{align*}
By \cref{prop:Coskun_OG}, $j^*\sigma_{2,1}$ and $j^*\sigma_{3,2,1,1}$ are nonzero effective 3-cycles which intersect non-positively with $\ch_3$, giving that $\OGr_+(k,2k)$ does not satisfy $\mathfrak{F}_3$ for $k=4,5$.
\end{proof}

%%%
\subsection{Linear sections in orthogonal Grassmannians}
%%%
We consider a linear section $X$ of $\OGr_+(k,2k)$ of codimension $c$, i.e.
%$$X = \OGr_+(k,2k) \cap (d_1\,H) \cap \ldots \cap (d_c\,H),$$ where $H= \frac{\sigma_1}{2}$ is a hyperplane section of the half-spinor embedding of $\OGr_+(k,2k)$. 
$$X = \OGr_+(k,2k) \cap H_1 \cap \ldots \cap H_c,$$ where $H_i \sim \frac{\sigma_1}{2}$ is a hyperplane section of the half-spinor embedding of $\OGr_+(k,2k)$. 
By \cite[Proposition 34]{AC13}, $X$ satisfies $\mathfrak{F}_2$ if and only if $c<4$. %and all $d_i=1$; thus we write $X= \OGr_+(k,2k) \cap H^c$ with $c=1,2,3$.
By \cref{thm:LM_Hx} the family of minimal rational curves on $\OGr_+(k,2k)$ is $\Gr(2,k)$.
Since these minimal rational curves are lines under the half-spinor embedding, 
the family of minimal rational curves on $X$ is 
%\textcolor{purple}{(add argument)}
$$H_x= \Gr(2,k) \cap L^c$$ where $L$ is a hyperplane section under the Pl{\"u}cker embedding. 

\begin{prop} \label{prop:F3_ciOG+}
For $k=5$ and $c<4$, $\OGr_+(5,10) \cap H^c$ does not satisfy the condition $\mathfrak{F}_3$.
\end{prop}
\begin{proof}
This follows from \cref{thm:3Fano_H2Fano} as $H_x=\Gr(2,5) \cap L^c$ does not satisfy $\mathfrak{F}_2$ by \cite[Proposition 32 (iv)]{AC13}.
\end{proof}

\subsection{Type C: symplectic Grassmannian}
%%%
\begin{prop}
For $2 \leq k < \frac{n}{2}$, $\SGr(k,n)$ does not satisfy the condition \Fthree.
\end{prop}

\begin{proof}
Let $X=\SGr(k,n)$. By \cref{thm:3Fano_H2Fano} $X$ does not satisfy \Fthree\, as $\rho(H_x)>1$.
\end{proof}

Now we consider $\SGr(k,2k)$ and show that in fact this also fails to satisfy the condition \Fthree\,. First, we consider the following result by Coskun.

\begin{prop}
[{\cite[Corollary 3.38]{Cos13}}]
\label{prop:Coskun_SG}
Let $i:\SGr(k,2k) \hookrightarrow \Gr(k,2k)$ be the natural inclusion, and $\sigma_{\lambda_1,\dots,\lambda_j}$ a Schubert cycle in $\Gr(k,2k)$. Then
%\begin{itemize}
%    \item[(1)] 
$i^*\sigma_{\lambda_1,\dots,\lambda_j}=0$ unless $k+1-j \ge \lambda_j$ for all $j$ with $1 \le j \le k$.
%   \item[(2)] 
%    $i^*\sigma_{k,k-1,\dots,2,1}\neq 0$.
%\end{itemize}
\end{prop}

\begin{cor}
\label{cor:Coskun_SG}
$i^*\sigma_{k,k-1,\dots,2,1}\neq 0$.
\end{cor}
\begin{proof}
By \cref{prop:Coskun_SG}, the condition for a  Schubert cycle to be nonzero on $\SGr(k,2k)$ is that $k\ge \lambda_1, k-1 \ge \lambda_2,\ldots,2 \ge \lambda_{k-1},1 \ge \lambda_k$. Therefore the only codimension $\frac{k(k+1)}{2}$ Schubert cycle whose restriction could be nonzero is $\sigma_{k,k-1,\ldots,2,1}$. By non-degeneracy of the intersection pairing in cohomology on $\Gr(k,2k)$, we have necessarily that $i^*\sigma_{k,k-1,\dots,2,1}\neq 0$.
\end{proof}

We recall from \cite[\S 6.3]{AC13} the expression of the third Chern character in terms of restrictions of Schubert cycles from $\Gr(k,2k)$:
    $$\ch_3(\SGr(k,2k))=\frac{-k+1}{6}
    i^*\sigma_3-\frac{-k+4}{6}i^*\sigma_{2,1}+ \frac{-k+7}{6}i^*\sigma_{1,1,1}.$$
    
\begin{lem} \label{lem:SG_relation}
For $X=\SGr(k,2k)$ with $k\ge 3$, we have $i^*\sigma_3=i^*\sigma_{1,1,1}$ and
$$\ch_3(X)= \frac{-k+4}{2}i^* \bigg(\sigma_3 - \frac{1}{6}\sigma_1^3 \bigg).$$
\end{lem}

\begin{proof}
By~\cite[Claim 35]{AC13} the equality $i^*\sigma_2=
i^*\sigma_{1,1}=
i^*\frac{1}{2}\sigma_1^2$ holds on $X$. Applying Pieri's formula we have
\begin{align*}
\frac{1}{2}i^*\sigma_1^3 
& = i^*(\sigma_1 \cdot \sigma_2) 
\stackrel{\text{Pieri}}{=} i^*\sigma_{2,1}+i^*\sigma_3
\\
\frac{1}{2}i^*\sigma_1^3
& = i^*(\sigma_1\cdot \sigma_{1,1})
\stackrel{\text{Pieri}}{=} i^*\sigma_{2,1}+i^*\sigma_{1,1,1}.
\end{align*}
This implies that $i^*\sigma_3=i^*\sigma_{1,1,1}$ and
$$
\ch_3(X)
=\frac{-k+1}{6}i^*\sigma_3-\frac{-k+4}{6} 
i^*\bigg( \frac{1}{2} \sigma_1^3 - \sigma_3 \bigg)
+ \frac{-k+7}{6}i^*\sigma_{3} 
= \frac{-k+4}{2} i^* \bigg(\sigma_3 - \frac{1}{6} \sigma_1^3 \bigg).
$$
\end{proof}

\begin{prop} \label{prop:F3_SG}
$\SGr(k,2k)$ does not satisfy the condition $\mathfrak{F}_3$. 
\end{prop}

\begin{proof}

\begin{comment}
%When $k=3$, the dimension of $\SGr(3,6)$ is $6$.
For $k=3$, we consider the cycle $\sigma_3$ on $\Gr(3,6)$. The restriction $i^*\sigma_3$ to $\SGr(3,6)$ is not zero since $\sigma_3 \cdot \sigma_{2,1}= \sigma_{3,2,1}$ and ${i^*\sigma_{3,2,1}} \neq 0$ by \cref{cor:Coskun_SG}; moreover, it is effective as non-negative linear combination of effective cycles on $\SGr(3,6)$, by \cite[\S 1]{Pragacz_Add}.
One computes that 
$$
\ch_3(\SGr(3,6)) \cdot \sigma_3
= \frac{1}{2}  \sigma_3 \cdot \sigma_3 - \frac{1}{12} \sigma_1^3 \cdot \sigma_3
= -\frac{1}{12} \sigma_1^3 \cdot \sigma_3 < 0
$$
as $\sigma_3 \cdot \sigma_3
\stackrel{\text{\cref{lem:SG_relation}}}{=}
\sigma_{1,1,1} \cdot \sigma_3 
\stackrel{\text{Pieri}}{=} 0$. Therefore $\SGr(3,6)$ is not $\mathfrak{F}_3$.

For $k=4$, $\ch_3(\SGr(4,8))=0$, hence $\SGr(4,8)$ does not satisfy $\mathfrak{F}_3$. 

For $k \geq 5$, 
\end{comment}

We consider the cycle $\rho \coloneqq \sigma_{k-1,k-2,k-3,k-3,k-4,\ldots,2,1}$ on $\Gr(k,2k)$: it has codimension $\frac{k(k+1)}{2}-3$ and its restriction $i^*\rho$ to $\SGr(k,2k)$ is effective as it is a non-negative linear combination of effective cycles on $\SGr(k,2k)$ by \cite[\S 1]{Pragacz_Add}. By the Pieri rule and \cref{prop:Coskun_SG} we have
$$i^*( \sigma_3 \cdot \rho)= i^*\sigma_{k,k-1,\dots,2,1},  \quad \quad  i^*( \sigma_1^3 \cdot \rho)= 6\cdot i^*\sigma_{k,k-1,\dots,2,1}.$$
It follows that $i^* \rho \neq 0$ as $i^*\sigma_{k,k-1,\dots,2,1} \neq 0$ by \cref{cor:Coskun_SG}, and 
\begin{align*}
\ch_3(\SGr(k,2k)) \cdot i^*\rho 
& \stackrel{\text{\cref{lem:SG_relation}}}{=} \frac{-k+4}{2}\, \bigg(\sigma_3 - \frac{1}{6}\sigma_1^3 \bigg) \cdot i^*\rho \\
& = \frac{-k+4}{2}\, i^*\bigg( \sigma_{k,k-1,\dots,2,1} - \frac{1}{6} 6 \cdot \sigma_{k,k-1,\dots,2,1}\bigg) 
= 0.
\end{align*}
Therefore $\SGr(k,2k)$ does not satisfy $\mathfrak{F}_3$. 
\end{proof}

%%%
\section{Rational homogeneous varieties of exceptional type}
\label{sec:hom_exceptional}
%%%
In this section, we show that none of the exceptional groups, when quotiented by a maximal parabolic subgroup, satisfies the condition \Fthree, and we study when the condition $\mathfrak{F}_2$ is satisfied. 
We recall the Dynkin diagrams of exceptional type and mark in black the short roots, i.e. the roots such that there is an arrow in the diagram pointing in its direction.

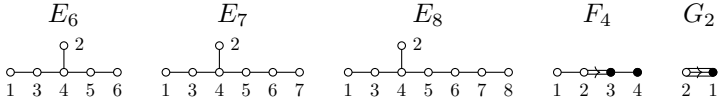
\begin{center}
\begin{tabular}{ccccc}
$E_6$ & $E_7$ & $E_8$ & $F_4$ & $G_2$ \\
\begin{dynkinDiagram}[mark=o,label]E6
\end{dynkinDiagram}
& \begin{dynkinDiagram}[mark=o,label]E7
\end{dynkinDiagram}
& \begin{dynkinDiagram}[mark=o,label]E8
\end{dynkinDiagram}
& \begin{dynkinDiagram}[mark=o,label]F4
\dynkinRootMark{*}3
\dynkinRootMark{*}4
\end{dynkinDiagram}
& \begin{dynkinDiagram}[mark=o,label,ordering=Carter]G2
\dynkinRootMark{*}1
\end{dynkinDiagram}
\end{tabular}
\end{center}

%For a parabolic subgroup $P^\alpha$ corresponding to a  non-short root $\alpha$, Landsberg and Manivel consider $X=G/P^\alpha$ and describe $H_x$ as follows.

%\begin{thm}[{\cite[Theorem 4.8]{LM03}, see also \cite[Theorem 2.5]{LM04} and the subsequent paragraph}] \label{thm:LM_Hx}
%Let $X=G/P^\alpha$ be a rational homogeneous variety such that $\alpha$ is not short. Then $H_x$ is homogeneous and the associated marked diagram is determined as follows: remove the node corresponding to $\alpha$ and mark the nodes adjacent to $\alpha$. 
%Moreover, the diagram is simply laced, i.e., without multiple edges, if and only if the embedding of $H_x$ in $\PP(T_x X)$ is minimal.
%\end{thm}
%\smallskip

\begin{prop}
$E_n/P^\alpha$ with $\alpha \neq 1,2,n$ and $F_4/P^2$ do not satisfy $\mathfrak{F}_2$.
\end{prop}
\begin{proof}
Let $X$ be one of the homogeneous spaces in the statement. By \cref{thm:LM_Hx}, $H_x$ is a product and thus has Picard rank $>1$. To show that $X$ does not satisfy $\mathfrak{F}_2$, we check that the polarized variety $(H_x, L_x)$ is not isomorphic to any of the exceptional pairs $(a)-(e)$ from the list in \cref{thmAC12:HxLx}:
\begin{itemize}
\item[-]
For $X = E_n/P^\alpha$ and $\alpha = 3,5$, we have $H_x \cong \Gr(2,k) \times \PS^l$, for some $k = 5,6,7$ and $l=1,2,3$. 
%\textcolor{red}{It is clear that $H_x$ is not (a), (b) or (c), since those are toric and Grassmannians are not. MISSING ARGUMENT for (d) and (e).}
\item[-]
For $X = E_n/P^4$, we have $H_x \cong \PS^2 \times \PS^1 \times \PS^{n-4}$.%, hence Picard rank of $H_x$ is three.
\item[-]
For $X = E_n/P^6$, $n=7,8$, we have $H_x \cong D_5/P^5 \times \PS^{n-6} \cong \OGr_+(5,10) \times \PS^{n-6}$. %It is not toric, hence not (a), (b) or (c), and for $n=8$, it is even-dimensional, hence not (d) or (e) either. 
%\textcolor{red}{MISSING ARGUMENT for $n=7$ (d) and (e).}
\item[-]
For $X = E_8/P^7$, we have $H_x \cong E_6/P^6 \times \PS^1$.
\item[-]
For $X = F_4/P^2$, we have $H_x \cong \PS^1 \times \PS^2$. By \cref{thm:LM_Hx}, 
the embedding of $H_x$ in $\PP(T_x X)$ is not minimal, and thus $L_x \neq \mathcal{O}(1,1)$.
%\textcolor{red}{Since $F_4$ is not simply laced, the embedding is not minimal, so $L_x \neq \mathcal{O}(1,1)$ and $H_x$ does not appear in \cref{thmAC12:HxLx}.}
\end{itemize}
Thus, $X$ does not satisfy $\mathfrak{F}_2$. 
\end{proof}

%%%
\subsection{Type E}
\subsubsection{Parabolic groups $P^1$, $P^2$}

\begin{prop} \label{prop:En/P1}
%The homogeneous variety 
$E_n/P^1$ satisfies the condition $\mathfrak{F}_2$ but not \Fthree\, for $n=6,7,8$.
\end{prop}

\begin{proof}
Let $X=E_n/P^1$. By \cref{thm:LM_Hx} we have $H_x= D_{n-1}/P^{n-1} = \OGr_+(n-1,2(n-1))$ and $L_x$ is a generator of $\Pic(H_x)$. Write $k =n-1$. We have $d=\dim(H_x)=\frac{k(k-1)}{2}$ and $\frac{\sigma_1}{2} \sim L_x$, as $\Pic(H_x) = \mathbb{Z} [\frac{\sigma_1}{2}]$, 
    where we denote by the same symbol the Schubert cycle $\sigma_1$ in $\Gr(k,2k)$ and its restriction to $\OGr_+(k,2k)$. 

%where $\sigma_1$ is the restriction from $\Gr(k,2k)$
%of the Schubert cycle corresponding to the Young diagram $(1)$.

As $b_4(X)=1$ by \cref{lem:Betti numbers}, we apply \cref{thmAC12:HxLx} and conclude that $X$ satisfies $\mathfrak{F}_2$ as 
$$-2K_{H_x}-dL_x= 2 \cdot 2(k-1)L_x - \frac{k(k-1)}{2}L_x= \frac{(k-1)(8-k)}{2}L_x$$ is ample.
%
%By \cite[\S 2.1.2]{LM04}, the embedding of $H_x$ in $\PP(T_x X)$ is minimal, hence $L_x$ is a generator of $\Pic(H_x)$, and $\frac{\sigma_1}{2} \cong L_x$. 
We compute
\begin{align*}
T(\ch_3(X))
 & \stackrel{\text{\eqref{equ:ch_2H}}}{=}  \ch_2(H_x) - \frac{1}{2}\bigg( c_1(H_x) - \frac{d}{2} c_1(L_x) \bigg) L_x - \frac{d-4}{12}L_x^2 \\
 & = 
 2L_x^2 -\frac{1}{2} \bigg( 2(k-1) - \frac{k(k-1)}{4} \bigg) L_x^2 - \frac{k(k-1)-8}{24} L_x^2 \\
 %& \stackrel{\text{Pieri}}{=} \bigg( \frac{n-4}{2}- \frac{n}{2} + \frac{3(n-3)}{4} - \frac{3n-13}{12} \bigg) \sigma_1^2 - \bigg( \frac{n-8}{2} + \frac{n-4}{2} \bigg) \sigma_{1,1} \\
 & = \frac{(k-5)(k-8)}{12} L_x^2.
\end{align*}
%$$
%T(\ch_2(X))
% \stackrel{\text{\eqref{equ:ch_1H}}}{=} 
% c_1(H_x) - \frac{d}{2}c_1(L_x) 
% = (k-1) 2L_x - \frac{k(k-1)}{2} L_x = - \frac{(k-1)(k-4)}{2} L_x,
%$$
which implies that $\ch_3(X)$ is not positive, so $X$ does not satisfy $\mathfrak{F}_3$.
\end{proof}

\begin{rem}
As $E_6/P^6 \cong E_6/P^1$, \cref{prop:En/P1} holds for $E_6/P^6$ as well.
\end{rem}

\begin{prop}
%The homogeneous variety 
$E_n/P^2$ satisfies the condition $\mathfrak{F}_2$ but not \Fthree \, for $n=6,7,8$.
\end{prop}

\begin{proof}
Let $X=E_n/P^2$. By \cref{thm:LM_Hx} we have $H_x= A_{n-1}/P^3 = \Gr(3,n)$ and and $L_x$ is a generator of $\Pic(H_x)$,  therefore $d=\dim(H_x)=3(n-3)$ and $\sigma_1 \sim L_x$. 
As $b_4(X)=1$ by \cref{lem:Betti numbers}, we apply \cref{thmAC12:HxLx} and conclude that $X$ satisfies $\mathfrak{F}_2$ as $-2K_{H_x}-dL_x= 2n L_x - 3(n-3)L_x= (9-n)L_x$ is ample.

For $n=8$, $H_x$ does not satisfy $\mathfrak{F}_2$, hence $X$ does not satisfy \Fthree \, by \cref{thm:3Fano_H2Fano}. From now on we assume that $n=6,7$.
We have
\begin{align*}
%T(\ch_2(X))
% & \stackrel{\text{\eqref{equ:ch_1H}}}{=} c_1(H_x) - \frac{d}{2}c_1(L_x) = n \sigma_1 - \frac{3(n-3)}{2} \sigma_1 = \frac{9-n}{2} \sigma_1 \\
 T(\ch_3(X))
 & \stackrel{\text{\eqref{equ:ch_2H}}}{=}  \ch_2(H_x) - \frac{1}{2}\bigg( c_1(H_x) - \frac{d}{2} c_1(L_x) \bigg) L_x - \frac{d-4}{12}L_x^2 \\
 & = \frac{n-4}{2}\sigma_2 - \frac{n-8}{2}\sigma_{1,1} -\frac{1}{2} \bigg( n - \frac{3(n-3)}{2} \bigg) \sigma_1^2 - \frac{3n-13}{12} \sigma_1^2 \\
 & \stackrel{\text{Pieri}}{=} \bigg( \frac{n-4}{2}- \frac{n}{2} + \frac{3(n-3)}{4} - \frac{3n-13}{12} \bigg) \sigma_1^2 - \bigg( \frac{n-8}{2} + \frac{n-4}{2} \bigg) \sigma_{1,1} \\
 & = \frac{-19+3n}{6} \sigma_1^2 - (n-6) \sigma_{1,1}.
\end{align*}
If $n=6$ we have $T(\ch_3(X))<0$, and if $n=7$ we have $T(\ch_3(X)) \cdot \sigma_{4,3,3}= -\frac{2}{3}<0$. In both cases, this implies that $\ch_3(X)$ is not positive, hence $X$ does not satisfy \Fthree.
\end{proof}

\subsubsection{Freudenthal variety}

The homogeneous variety $E_7/P^7$ is also known as the Freudenthal variety $G(\mathbb{O}^3,\mathbb{O}^6)$; it has dimension $27$ and index $18$. We refer for instance to ~\cite[\S 2.1, \S 2.3]{CMP} for more details on the geometry of $E_7/P^7$.

%The homogeneous variety $E_7/P^7$ is also known as the Freudenthal variety $G(\mathbb{O}^3,\mathbb{O}^6)$. It is the closed $E_7$-orbit in $\mathbb{P}V$, where $V$ is an irreducible representation of $E^7$ of dimension 56; it is a Fano manifold of dimension $27$. We refer for instance to Section 2.3.2 of ~\cite{CMP} for more details on the geometry of the variety $E_7/P^7$.

\begin{prop}
\label{prop_E7/P7}
%The homogeneous variety 
$E_7/P^7$ satisfies the condition $\mathfrak{F}_2$ but not \Fthree.
\end{prop}

\begin{proof}
Let $X=E_7/P^7$. By \cref{thm:LM_Hx} we have $H_x= E_{6}/P^6$ and $L_x$ is a generator of $\Pic(H_x)$, therefore $d=\dim(H_x)=16$ and $L_x \sim H$, as the hyperplane section $H$ in $H_x$ is a generator of $\Pic(H_x)$ by \cite[Proposition 5.1]{IM}.
By \cref{lem:Betti numbers} $b_4(X)=1$, hence we apply \cref{thmAC12:HxLx} and obtain that $X$ satisfies $\mathfrak{F}_2$ as 
$$2T(\ch_2(H_x))=-2K_{H_x}-dL_x= 24H - 16H= 8H$$ is ample.
We apply \cref{prop:chH} to compute:
\begin{align*}
 %  T(\ch_2(X))
 %   & \stackrel{\text{\eqref{equ:ch_1H}}}{=}\ch_1(H_x) - \frac{d}{2} c_1(L_x)
 %  = 12H - 8H= 4H \\
    T(\ch_3(X))
    & \stackrel{\text{\eqref{equ:ch_2H}}}{=} \ch_2(H_x) - \frac{1}{2}\bigg( c_1(H_x)-\frac{d}{2} c_1(L_x) \bigg) \cdot L_x -  \frac{d-4}{12} L_x^2 \\
    & \stackrel{\text{\cref{lem:E6/P6}}}{=} 3H^2  - \frac{1}{2} 4H^2 - H^2  =0.
\end{align*}
%where $\ch_2(H_x)=3H^2$ is obtained as follows. As described for instance in \cite{IM}, $E_6/P^6$ admits an embedding in $\P^{26}$, and the normal bundle $\mathcal{N}$ of $H_x$ in $\P^{26}$ has $c_1(\mathcal{N})=15H$ and $c_2(\mathcal{N})=102H^2$ by \cite[Proposition 7.1]{IM}, which gives
%$$\ch_2(H_x) =\ch_2(\P^{26})_{|H_x}-\ch_2(\mathcal{N})=3H^2.$$
We conclude that $X$ does not satisfy \Fthree.
\end{proof}

\begin{lem} \label{lem:E6/P6}
$\ch_2(E_6/P^6)= 3H^2$ and $\ch_3(E_6/P^6)=0$.
\end{lem}
\begin{proof}
Let $X=E_6/P^1=E_6/P^6$; as described for instance in \cite{IM}, $X$ admits an embedding in $\P^{26}$. Denote by $\mathcal{N}$ the normal bundle of $X$ in $\P V \cong \P^{26}$; by \cite[Proposition 7.1] {IM} we have $c_1(\mathcal{N})=15H$, $c_2(\mathcal{N})=102H^2$, and $c_3(\mathcal{N})=414 H^3$, where $H$ denotes a hyperplane section. Then we obtain
\begin{align*}
\ch_2(X) &=\ch_2(\P^{26})_{|X}-\ch_2(\mathcal{N})=3H^2, \\
\ch_3(X) &=\ch_3(\P^{26})_{|X}-\ch_3(\mathcal{N})=0.
\end{align*}
\end{proof}

\begin{prop}
%The homogeneous variety 
$E_8/P^8$ satisfies the condition $\mathfrak{F}_2$ but not \Fthree.
\end{prop}

\begin{proof}
Let $X=E_8/P^8$. By \cref{thm:LM_Hx} we have $H_1 \coloneqq H_x= E_{7}/P^7$ the Freundenthal variety, whose corresponding minimal family of rational curves through a general point is $H_2=E_6/P^6$.
%the Cayley plane. 
Thus, we have $d=\dim(H_1)=27$, $c_1(H_1)=18L_1$ and by \cite[\S 2.3]{CMP} $\Pic(H_i) = \mathbb{Z} [h_i]$ where $h_i$ denote a hyperplane class on $H_i$; by \cref{thm:LM_Hx} $h_i \sim L_i$. 

We have $b_4(X)=1$ by \cref{lem:Betti numbers}, hence we apply \cref{thmAC12:HxLx} and obtain that $X$ satisfies $\mathfrak{F}_2$ as $-2K_{H_1}-dL_1= 36L_1 - 27L_1=9 L_1$ is ample. We compute
\begin{align*}
%T(\ch_2(X))
 %& \stackrel{\text{\eqref{equ:ch_1H}}}{=} c_1(H_1) - \frac{d}{2}c_1(L_1) 
% = 18L_1  - \frac{27}{2} L_1 = \frac{9}{2} L_1 \\
 T(\ch_3(X))
 & \stackrel{\text{\eqref{equ:ch_2H}}}{=}  \ch_2(H_1) - \frac{1}{2}\bigg( c_1(H_1) - \frac{d}{2} c_1(L_1) \bigg) L_1 - \frac{d-4}{12}L_1^2 \\
 & = \ch_2(H_1) - \frac{9}{4}L_1^2 -\frac{23}{12}L_1^2 = \ch_2(H_1) - \frac{25}{6}L_1^2 \\
 T \circ T(\ch_3(X)) 
 & = T(\ch_2(H_1)) - \frac{25}{6}T(L_1^2) \stackrel{\text{Prop \eqref{prop_E7/P7}}}{=} 
 4L_2 - \frac{25}{6}T(L_1^2) \\
 & \leq 4L_2 - \frac{25}{6}L_2 = - \frac{1}{6}L_2
\end{align*}
where the last inequality holds by \cite[Lemma 2.7 (1)]{AC12}.
We conclude that $X$ does not satisfy \Fthree.
\end{proof}

%%%
\subsection{Type F}
\begin{prop}
$F_4/P^1$ does not satisfy the condition $\mathfrak{F}_2$.
\end{prop}

\begin{proof}
Let $X=F_4/P^1$. By \cref{thm:LM_Hx} we have $H_x= C_{3}/P^3=\SGr(3,6)$, and the embedding of $H_x$ in $\PP(T_x X)$ is not minimal. Therefore $\Pic(H_x)$ is not generated by $[L_x]$. Since this pair $(H_x,L_x)$ is not in the exceptional list of \cref{thmAC12:HxLx}, we conclude that $X$ does not satisfy $\mathfrak{F}_2$.
%hence $d=\dim(H_x)=6$, and we have that the embedding of $H_x$ in $\PP(T_x X)$ is not minimal, hence $L_x \geq 2\sigma_1$, where $\sigma_1$ is a generator of $\Pic(H_x)$.
%As $b_4(X)=1$, see for a reference \cite[\S 5.5]{AC12}, we apply \cref{thmAC12:HxLx} and obtain that $X$ does not satisfy $\mathfrak{F}_2$ as $-2K_{H_x}-dL_x \leq  8\sigma_1 -12 \sigma_1=-4 \sigma_1$ is not ample.
\end{proof}

\begin{prop}
$F_4/P^3$ does not satisfy the condition $\mathfrak{F}_2$.
\end{prop}

\begin{proof}
Let $X=F_4/P^3$. By \cite[Proposition 6.9]{LM03} $H_x$ is a nontrivial $Q^4$-bundle over $\PP^1$, in particular $\rho(H_x)\neq 1$. Note that none of the exceptional pairs $(H_x,L_x)$ in \cref{thmAC12:HxLx} admit a nontrivial $Q^4$-fibration. We conclude that $X$ does not satisfy $\mathfrak{F}_2$.
%hence $d=\dim(H_x)=5$ and $\rho(H_x)\neq 1$. Suppose by contradiction that $X$ satisfies $\mathfrak{F}_2$. By \cref{thmAC12:HxLx} we have that $H_x$ is isomorphic to one of cases $(b)-(e)$ in the list. In all those cases, none of the extremal rays of  the nef cone has an Iitaka fibration with target $\PP^1$, which is a contradiction. We conclude that $X$ does not satisfy $\mathfrak{F}_2$.
\end{proof}

\begin{prop}
$F_4/P^4$ satisfies the condition $\mathfrak{F}_2$ but not \Fthree.
\end{prop}

\begin{proof}
By \cite[proof of Proposition 6.5]{LM03} $F_4/P^4$ is the generic hyperplane section of $E_6/P^6$. We write $X=E_6/P^6 \cap H$, $Y=E_6/P^6$ and recall that $\Pic(X)= \mathbb{Z}[H]$. 
From \cref{lem:E6/P6} and \eqref{equ:chk_ci}, we obtain
\begin{align*}
    \ch_2(X)
    & =
    %\stackrel{\text{\eqref{equ:chk_ci}}}{=} 
    \bigg(\ch_2(Y) - \frac{1}{2}H^2 \bigg)_{|Y}
    = \bigg(3H^2- \frac{1}{2}H^2 \bigg)_{|Y} = \frac{5}{2} H^2_{|x}\\
    \ch_3(X)
    & = 
    %\stackrel{\text{\eqref{equ:chk_ci}}}{=} 
    \bigg(\ch_3(Y) - \frac{1}{6}H^3 \bigg)_{|Y}=  - \frac{1}{6}H^3_{|X}.
\end{align*}
This implies that $X=F_4/P^4$ satisfies $\mathfrak{F}_2$ but not \Fthree.
\end{proof}

%%%
\subsection{Type G}
%\subsubsection{G2/P1}

\begin{prop}
$G_2/P^1$ satisfies the condition $\mathfrak{F}_2$ but not \Fthree.
\end{prop}

\begin{proof}
By \cite[\S 6.1]{LM03}, we have that $G_2/P^1= Q^5 \subset \mathbb{P}^6$. Applying \cref{thm:2Fano_highindex} and \cref{equ:chk_ci}, we conclude that $G_2/P^1$ satisfies $\mathfrak{F}_2$ but not \Fthree.
\end{proof}

\begin{prop} \label{prop:F3_G2}
$G_2/P^2$ satisfies the condition $\mathfrak{F}_2$ but not \Fthree.
\end{prop}

\begin{proof}
Let $X=G_2/P^2$. Then $X$ satisfies $\mathfrak{F}_2$ by \cref{thm:2Fano_highindex}. Moreover, we have
\begin{align*}
  6 \cdot \ch_3(X)\cdot c_1^2
  &= (c_1^3-3\,c_1\,c_2+3\,c_3 )\cdot c_1^2 =c_1^5-3\,c_1^3\,c_2+3\,c_3\,c_1^2\\
  & =4374 -3\cdot 2106 + 3 \cdot 594 = -162
\end{align*}
where the Chern numbers are computed in \cite[Table 1]{KT20}. As $\frac{c_1^2}{27}$ is an integral class by \cite[\S 2]{KT20}, we obtain $\ch_3(X)\cdot \frac{c_1^2}{27}=-1$, so that $X$ does not satisfy \Fthree.
\end{proof}

%%%

%%%
\section{Higher Fano manifolds with high index}
\label{sec:high_index}
%\subsubsection{G2/P1}

In \cite{AC13} Araujo and Castravet classify $2$-Fano manifolds of high index.
We note that there are inaccuracies in the dimension bounds in the original statement of \cite[Theorem 3]{AC13}. 
We state their classification with the correct bounds.

\begin{thm}[{\cite[Theorem 3]{AC13}}]
\label{thm:2Fano_highindex}
Let $X$ be a Fano manifold of dimension $n\ge 3$ and index $i_X\ge n-2$. If $X$ satisfies ${\mathfrak {F}_2}$, then $X$ is isomorphic to one of the following.
\begin{itemize}
\item $\PP^n$.

\item Complete intersections in projective spaces:
\begin{itemize}
\item[-]  Quadric hypersurfaces $Q^n\subset \PP^{n+1}$ with $n>2$;
\item[-]  Complete intersections of quadrics $X_{2\cdot2}\subset\PP^{n+2}$ with  $n>5$;
\item[-]  Cubic hypersurfaces $X_3\subset\PP^{n+1}$ with $n>7$;
\item[-]  Quartic hypersurfaces in $\PP^{n+1}$ with $n>14$;
\item[-]  Complete intersections $X_{2\cdot3}\subset\PP^{n+2}$ with $n>10$; 
\item[-]  Complete intersections $X_{2\cdot2\cdot2}\subset\PP^{n+3}$ with $n>8$.  
\end{itemize}

\item Complete intersections in weighted projective spaces:
\begin{itemize}
\item[-]  Degree $4$ hypersurfaces in $\PP(2,1,\ldots,1)$ with $n>11$; 
\item[-]  Degree $6$ hypersurfaces in $\PP(3,2,1,\ldots,1)$ with $n>23$; 
\item[-]  Degree $6$ hypersurfaces in $\PP(3,1,\ldots,1)$ with $n>26$; 
\item[-] Complete intersections of two quadrics in $\PP(2,1,\ldots,1)$ with $n>2$. 
\end{itemize}

\item $\Gr(2,5)$.
\item $\OGr_+(5,10)$ and its linear sections of codimension $c<4$. 
\item $\SGr(3,6)$.
\item $G_2/P^2$.
\end{itemize}
\end{thm}

We recall that by \cite[Lemma 8]{AC13}, if $Y$ is a smooth variety and $X$ is a smooth complete intersection of divisors $D_1,\ldots,D_c$ in $Y$, then %the $k$-th Chern character is given by
\begin{equation} \label{equ:chk_ci}
    \ch_k(X)=  \left( \ch_k(Y) - \frac{1}{k!} \sum_{i=1}^{c} D_i^k \right)_{|X}.
\end{equation}

In particular, let $X$ be a smooth complete intersection of hypersurfaces of degrees $d_1,\ldots, d_c$ in $\PP^n$ and denote by $h:=c_1(\mathcal{O}_{\PP^n}(1))$ the hyperplane class in $\PP^n$. Then 
$$\ch_k(X)=\frac{1}{k!}\big((n+1)-\sum d_i^k\big)\,h^k_{|X},$$
hence $X$ satisfies ${\mathfrak{F}_k}$ if and only if $\sum d_i^k\leq n$. 

Let $\PP(a_0,\ldots, a_n)$ denote the weighted projective space with gcd$(a_0,\ldots, a_n)=1$, and $H$ be the effective generator of its class group. Let $X$ be a smooth complete intersection of hypersurfaces in $\PP(a_0,\ldots, a_n)$ with classes $d_1H,\ldots, d_cH$. Assume that $X$ is smooth, and contained in the smooth locus of $\PP(a_0,\ldots, a_n)$. Then
$$\ch_k(X)=\frac{a_0^k+\ldots+a_n^k-\sum d_i^k}{k!}c_1(H_{|X})^k;$$
it follows that $X$ satisfies ${\mathfrak{F}_k}$ if and only if $\sum_{i=1}^{c} d_i^k < \sum_{j=0}^{n} a_j^k$. 
\vspace{10pt}

\begin{proof}[Proof of Theorem~\ref{thm:3Fanos_highindex}]
The list in the theorem is obtained starting from \cref{thm:2Fano_highindex} and studying the positivity of the third Chern character case by case. For complete intersections we apply \cref{equ:chk_ci} and we determine sharp bounds on $n$. The remaining $2$-Fano manifolds do not satisfy the condition \Fthree \, by Propositions~\ref{prop:F3_A}, \ref{prop:F3_OG+}, \ref{prop:F3_ciOG+}, \ref{prop:F3_SG} and \ref{prop:F3_G2}.
\end{proof}

\end{document}